\documentclass[11pt]{amsart}

\usepackage[utf8]{inputenc}

\usepackage{graphicx}
\usepackage{subcaption}
\usepackage{amsmath}
\usepackage{amssymb}
\usepackage{amsfonts}
\usepackage{amsthm}
\usepackage{enumerate}
\usepackage{lscape}
\usepackage{dsfont}
\usepackage{color}
\usepackage{mathtools}

\usepackage{setspace}

\newtheorem{theor}{Theorem}[section]

\newtheorem{dfn}{Definition}

\newtheorem{prop}[theor]{Proposition}
\newtheorem{corol}[theor]{Corollary}
\newtheorem{lem}[theor]{Lemma}
\newtheorem{guess}{Conjecture}

\theoremstyle{remark}
\newtheorem{rem}[theor]{Remark}

\newcommand{\Z}{\mathds{Z}}

\newcommand{\CP}{\mathds{C}\mathrm{P}}
\newcommand{\de}{\partial}

\newcommand{\inv}{${S}^1$-invariant }

\newcommand{\R}{\mathbb{R}}

\newcommand{\C}{\mathbb{C}}

\newcommand{\N}{\mathbb{N}}

\newcommand{\K }{K\"{a}hler\ }
\newcommand{\KE }{K\"{a}hler-Einstein\ }

\title[Lower dimensional KE metrics via integrable structures]{Lower dimensional ${S}^1$-invariant K\"{a}hler-Einstein metrics via
integrable structures}
\author{Filippo Salis}
\begin{document}
\maketitle

\begin{abstract}{We focus on the classical open problem of the classification of \KE manifolds that can be \K immersed into a complex projective space endowed with the Fubini-Study metric. In particular, we will deal with such problem in the special case of \KE metrics admitting symmetries of rotational type. This leads to certain integrable distributions allowing a classification of such metrics.}\end{abstract}

\tableofcontents

\section{Introduction}
The existence and rigidity of holomorphic and isometric immersions (\K immersions) of a given \K  manifold into a space of constant holomorphic sectional curvature (called complex space forms) can be considered as a classical and long-staying problem in complex differential geometry (see e.g. \cite{bochner, Cal}). Nevertheless, a complete classification of \K manifolds admitting such type of immersions does not exist, even for \K manifolds of great interest such as \KE manifolds.

In \cite{umehara2}, M. Umehara  proved that \KE manifolds that can be  \K immersed into a  complex space form with non-positive holomorphic sectional curvature, are only the totally geodesic submanifolds of the ambient space. Instead, when the space form has {positive} holomorphic curvature, only some partial results are known (see for instance \cite{smyth,tak,ch,ts,hano,hulin,hulinlambda}). They anyway made some authors  (see e.g. \cite[Chap. 4]{loizedda}) suppose the validity of the following
\begin{guess}
 A \KE manifold that can be \K immersed into a complex projective space, is  an  open subset of a  complex flag manifold, namely  a compact simply-connected \K manifold acted upon  transitively by its holomorphic isometry group. 
\end{guess}

We recall that complex projective spaces $\CP^N$ endowed with metrics\footnote{If $c=1$ we have the standard Fubini-Study metric $g_{FS}$.}  associated to \K forms reading in affine coordinates $z$ as $$\frac{i}{2c}\de\bar\de\log(1+\|z\|^2),$$ are the universal covering of  an $n$-dimensional complex space form with holomorphic sectional curvature equal to $4c>0$. 

In the present paper, we are going to test the above conjecture in the case of \inv \K manifolds, i.e. those \K manifolds admitting (around a suitable point and in suitable holomorphic coordinates)  a local  \K potential 
\begin{equation}\label{U}\Phi:0\in U\subseteq \C^n\to \R\end{equation} of the form $\Phi\left(|z_1|^2,\dots,|z_n|^2\right)$.

We are going to study only \K manifolds $(M,g)$ that can be \K immersed in complex projective spaces endowed with the standard Fubini-Study metric (in this case we say that $g$ is projectively induced).  By introducing suitable normalization coefficients, our proof can be easily adapted (leading to similar results) to cases where the ambient space is endowed with a different multiple of the Fubini-Study metric.

Since complex projective spaces are the only irreducible \inv flag manifolds and since only the integer multiples of the Fubini-Study metric  are projectively induced, the above conjecture reads in our setting as
\begin{guess}\label{conj}
The only projectively induced and \inv \KE manifolds are open subsets of $\CP^{n_1}\times\mathellipsis\times\CP^{n_k}$ endowed with the \K metric
$$q\left(c_1 g_{FS} \oplus \mathellipsis \oplus c_k g_{FS}\right), $$
where $k$ and $q\in\Z^+$. Let $G$ be the least common multiple between $( n_1+1,\mathellipsis, n_k+1)$,  then $$c_i=\frac{G}{n_i+1}\text{ for }i=1,\mathellipsis,k.$$
\end{guess}
\begin{rem}\label{embedding}

The homogeneous spaces  $\big(\CP^{n_1}\times\dots\times\CP^{n_k},q (c_1g_{FS}\oplus\dots\oplus  c_kg_{FS})\big)$ are  embedded into  $\CP^{{n_1+q c_1\choose q c_1}\cdots{n_k+q c_k\choose q  c_k}-1}$. A \K embedding can be explicitly described through a composition of suitable normalizations of  Veronese embeddings:
\begin{align*}
(\CP^{n},cg_{FS})& \to (\CP^{\binom{n+c}{c}-1},g_{FS}) \\
[Z_i]_{0\leq i\leq n}& \mapsto\sqrt\frac{(c-1)!}{c^{c-2}} \left[\frac{Z_0^{c_0}\mathellipsis Z_n^{c_n}}{\sqrt{c_0!\mathellipsis c_n!}} \right]_{c_0+\mathellipsis +c_n=c},
\end{align*}
together with a {Segre embedding}  (cfr. \cite{Cal,loizedda}). \end{rem}

Our main result is

\begin{prop}\label{toric}
A $n$-complex dimensional \KE manifold endowed with a \inv projectively induced metric is an open subset of a compact toric \K manifold (i.e. a \K manifold equipped with an effective Hamiltonian action of the $n$-dimensional real torus).
\end{prop}
\begin{rem}
A simple example of non-toric compact \inv \KE manifold can be the compact complex torus $\C^n/\Z^n$ together with the metric induced by the Euclidean metric on $\C^n$.
\end{rem} 

From Arezzo-Loi-Zuddas' classification of lower dimensional compact Fano toric  manifolds admitting a \KE projectively induced metric (see Subsection \ref{alz}), we get the following corollary, namely we proved that Conjecture \ref{conj} holds true for manifolds with dimension less than five.
\begin{corol}\label{main}
Let $M$ be a  projectively induced and \inv \KE manifold with complex dimension less than 5. Then $M$ is an open subsets of one of the following manifolds endowed with an integer multiple of the metric below indicated:
\begin{itemize}
\item if  $\dim M=2$, $(\CP^2,g_{FS}) \text{ or } (\CP^1\times \CP^1, g_{FS}\oplus g_{FS});$
\item if  $\dim M=3$, $(\CP^3,g_{FS}),\  (\CP^2\times \CP^1, 2g_{FS}\oplus 3g_{FS}) \text{ or }$ $ (\CP^1\times \CP^1 \times \CP^1, g_{FS}\oplus g_{FS}\oplus g_{FS});$
\item if  $\dim M=4$, $(\CP^4,g_{FS}),\  (\CP^3\times \CP^1, g_{FS}\oplus 2 g_{FS}),\  (\CP^2\times \CP^2, g_{FS}\oplus g_{FS}),\   (\CP^2\times \CP^1\times \CP^1, 2g_{FS}\oplus 3g_{FS}\oplus 3g_{FS}),\text{ or }(\CP^1\times \CP^1 \times \CP^1 \times \CP^1, g_{FS}\oplus g_{FS}\oplus g_{FS}\oplus g_{FS}).$
\end{itemize}
\end{corol}

The paper is organized as follow. We are going to recall the definition and some basic facts about Calabi's diastasis function and Bochner's coordinates that will be used in the proof of   Lemmas \ref{1} and \ref{2}. This two auxiliary lemmas are particularly important in the proof of compactness and simple connectedness of complete \KE \inv projectively induced manifolds (Lemma \ref{3}), topological properties that turn out to be fundamental in the construction of an integrable distribution, used in the proof of Proposition \ref{toric}. Finally, the proof of Corollary \ref{main} is given.

\section{Proof of Proposition \ref{toric} and Corollary \ref{main}}
 It is worth to point out that the condition of being a \KE metric reads locally, with respect to any local holomorphic coordinates, as a  nonlinear overdetermined system of fourth-order elliptic PDEs. From regularity results for elliptic PDEs, we have the real-analyticity  of  any \K potential  (see e.g. \cite{aubin}). 
 We easily see that we can assume without loss of generality that the \inv  \K potential $\Phi$ is the so-called Calabi's diastasis function (cfr. \cite{bochner, Cal, loizedda}), namely the only \K potential whose Taylor expansion does not contain zero and first order terms. Moreover, in our case,  coordinates $z=(z_1,\mathellipsis,z_n)$  can be chosen in order to have the following
\begin{equation*}
\Phi(z)=\sum_{\alpha=1}^n|z_\alpha|^2+\psi_{2,2},
\end{equation*}
where $\psi_{2,2}$ is a power series with degree $\geq 2$ in both $z$ and $\bar z$. 
Such coordinates are uniquely defined up to unitary transformations (see \cite{loizedda}) and they are called Bochner's coordinates  with respect to the origin of the open set $U\subseteq\C^n$, where is defined $\Phi$ (cfr. \eqref{U}).

\begin{lem}\label{1}
Let $\Phi$ be a \inv \K potential of a projectively induced \inv \K metric. Then
\begin{equation}\label{diastpol}
\Phi=\log\left(1+\sum_{j=1}^n|z_j|^2+\sum_{j=n+1}^N a_j|z^{m_j}|^2\right)
\end{equation}
where $a_j\in\R^+$, $m_i\in\N^n$ and $m_j\neq m_i$ for $j\neq i$.
\end{lem}
\begin{proof}
Let $f:(M,g)\to(\CP^N,g_{FS})$ be a \K immersion. We are going to assume without loss of generality that $f$ is full, namely $f(M)$ is not contained into any $\CP^h$ with $h<N$. Let $w=(z_1,\mathellipsis,z_{\dim M})$ be  Bochner's coordinates on $V\subseteq M$.
 By \cite{Cal}, there exists a   Bochner's coordinate system $(w_1,\mathellipsis,w_{ N})$ on $\CP^N$ such that
$$w_1|_{f(V)}= z_1, \mathellipsis, w_{ \dim M}|_{f(V)}= z_{\dim M},$$
and
$$ w_{\dim M+1}|_{f(V)}= f_{\dim M+1}(z), \mathellipsis, w_{ N}|_{f(V)}= f_{ N}(z),$$ where $f_j$ are holomorphic functions that vanish at the second order at the origin.

The hereditary property of Calabi's diastasis (see \cite{Cal}) states that Calabi's diastasis function $D^g_p$ for $g$ around a point $p\in M$ needs to satisfy
 $$D^g_p = D^{g_{FS}}_{f(p)}\circ f.$$
Therefore, we have that $\Phi$ needs to assume the following form
$$\Phi=\log\left(1+\sum_{j=1}^n|z_j|^2+\sum_{j=n+1}^N|f_j(z)|^2\right).$$
Since $\Phi$ is \inv and $f$ is full, we get our statement.
\end{proof}

\begin{lem}\label{2}
Let $(M,g)$ be a \KE \inv manifold. Its \inv \K potential satisfy the following Monge-Ampère equation
\begin{equation}\label{MA}
\det g=e^{-\frac{\lambda}{2}\Phi},\end{equation}
for a suitable real constant $\lambda$.
\end{lem}
\begin{proof}
Since $g$ is a \KE metric, there exists a suitable $\lambda\in \R$ such that 
$$-i\de\bar\de\log\det g=\frac{i\lambda}{2}\de\bar\de\Phi.$$
Hence, there exists a holomorphic function $f:U\to\C$, defined on a suitably small open subset $U$ of $M$,  such that
$$\det g=e^{-\frac{\lambda}{2}\Phi +f+\bar f}.$$
Once Bochner's coordinates are set, by comparing the series expansions of both sides of the previous equation, it follows $f+\bar f\equiv 0$.
\end{proof}

\begin{lem}\label{3}
A \KE manifold endowed with a \inv projectively induced metric is an open subset of  a compact and simply connected manifold.
\end{lem}
\begin{proof}
 Since the previous equation \eqref{MA} can be seen as an equality between polynomials in $z$ and $\bar z$, by comparing their degrees, we get $\lambda$  needs to be positive (for details see \cite{mannosalis}), namely $M$ is a  Fano manifold.

 D. Hulin proved in \cite{hulin} that
every \KE manifold \K immersed into a complex projective space can be extended to a  complete \KE manifold, \K immersed into the same complex projective space.
Therefore, by Bonnet-Myers' theorem, $M$ needs to be   {compact}. Moreover, every compact \K manifold with positive definite Ricci tensor is simply connected by \cite{koricci}.
\end{proof}

From now on, we are going to consider  without loss generality only metrics defined on complete, compact and simply connected manifolds.

\begin{proof}[Proof of Proposition~\ref{toric}]
Let $\phi:U\cap(\C\setminus\{0\})^n \to \R$ such that $$2\phi(\log{|z_1|^2},\mathellipsis, \log{|z_n|^2})=\Phi\left(|z_1|^2,\dots,|z_n|^2\right).$$ 
Chosen any branch of the complex logarithm,  we  set holomorphic coordinates $w_i=\log z_i$. Hence,  the \K form reads locally as\footnote{We are going to use the Einstein convention on repeated indices.}
$$\omega=i \de\bar \de \phi=i\frac{\de^2 \phi}{\de w_k \de \bar w_j}dw_k\wedge d\bar w_j=\frac{1}{2}\frac{\de^2 \phi}{\de r_k \de  r_j}dr_k\wedge d \theta_j,$$
where $w_k=r_k+i \theta_k$. Therefore,
$$g=\phi^{kj}(r) dr_kdr_j+\phi^{kj}(r)d\theta_k d \theta_j,$$
where $\phi^{kj}(r)=\frac{\de^2 \phi}{\de r_k \de  r_j}$. We easily see that $\de_{\theta_k}$ are Killing vector fields  on  $U$. Being $M$ simply connected and real analytic, each Killing vector field can be extended to a unique Killing vector field defined on the whole manifold $M$ (see \cite{nomizu}  theorems 1 and 2). We denote with $X_k$ the global extension of $\de_{\theta_k}$.  We recall that every Killing vector field on a compact \K manifold is real holomorphic (see e.g.  \cite{Mor} prop. 9.5), i.e.  if we identify $TM$   with $T^{(1,0)}M\subset TM\otimes\C$ in the usual way (namely by identifying $Z\in TM$ with $\tilde Z=(Z-i JZ)/2\in T^{(1,0)}M$), we have that $\tilde X_k=(X_k-i JX_k)/2$  is holomorphic.

Since $M$ is complete, we get a holomorphic and isometric action of $\R$ on $M$ by means of the flows of the each Killing vector field $X_k$. Furthermore, for any $1\leq k,j\leq n$, $[ X_k, X_j]$ is a Killing vector field vanishing on $U$. Since Killing vector fields (different from the identically zero vector field) vanish on totally geodesic submanifolds of real codimension at least 2 (cfr.  \cite{kob}), the commutator $[ X_k,  X_j]$ needs to vanish everywhere on $M$. Therefore, we have a transitive holomorphic and isometric action  $\mathcal{G}$ of $\R^n$ on $M$.

Let $V\subset M$ be the subset where at least a Killing vector field $X_k$ vanishes. By considering that $M\setminus V$ consists of and only of maximal dimensional orbits of the action $\mathcal{G}$, such action can be restricted to $M\setminus V$. We easily see that $\Z^n$ is the stabilizer of $\mathcal{G}$ in $(M\setminus V)\cap U$. Hence, we have a holomorphic and isometric action of the real torus  $\R^n/\Z^n$ on $M\setminus V$. Since  stabilizers of $\mathcal{G}$ in points belonging to $V$ contain $\Z^n$, this toric  action can be extended to the whole manifold $M$.

Every Killing vector field on a compact and simply connected \K manifold $M$, is Hamiltonian. Indeed, since these vector fields are also real holomorphic because $M$ is compact, they need  to be symplectic too, i.e. $i_{X_k}\omega$ is closed. Being $M$ simply connected, $H^1_{\text{dR}}(M)=0$. Therefore,  $i_{X_i}\omega$ needs to be also  exact. Therefore,  the action of $\R^n/\Z^n$ is Hamiltonian\footnote{ We easily see that the action is effective. Indeed, if $p\in(M\setminus V)\cap U$, then $g\cdot p\neq p$ for every $g\in \R^n/\Z^n$.}.

 The existence of such  action allow us to conclude that $M$ is a toric \K manifold.
\end{proof}

\begin{proof} [Proof of Corollary~\ref{main}] C. Arezzo, A. Loi, F. Zuddas proved in  \cite{ALZ} that the only compact toric Fano manifolds with dimension less than five, that can be endowed with a \KE projectively induced metric, are  (or can be decomposed as a product of) complex projective spaces with  multiples of the Fubini-Study metric.  By  \cite{Cal},  only integer multiples of the Fubini-Study metric are projectively induced. Therefore, we get the statement of  Corollary \ref{main} by straightforward computations of normalization coefficients necessary to obtain an Einstein metric on the product of complex projective spaces.\end{proof}

  In the following section we provide a sketch of the technique used by  C. Arezzo, A. Loi, F. Zuddas  to get their result.
\subsection{Arezzo-Loi-Zuddas' result}\label{alz}
 Let $H_k$ be a  real function on $M$  such that
$$i_{X_k}\omega=-dH_k, $$
where $X_k$ are symplectic vector fields related to the toric action on $M$ as described in the proof of Proposition \ref{toric}.
A  moment map $\mu:M\to\R^n$ for such action can be written  as $$\mu(x)=\left(H_1(x),\mathellipsis,H_n(x)\right).$$
By  Atiyah and Guillemin-Sternberg convexity theorems \cite{at, guil}, the imagine  of a compact symplectic manifold through  a moment map for a Hamiltonian action of a torus, is the convex hull\footnote{The smallest convex set of $\R^n$ containing such points.} of the images of the fixed points of the action. Moreover, T. Delzant proved  (see \cite{del} or \cite{don}) that momentum images of toric symplectic manifolds need to satisfy some additional properties, namely they turn out to be  the so-called Delzant polytopes. We recall that a polytope is a convex subset of $\R^n$ that can be represented as intersection of a finite number of half-spaces.
\begin{dfn}
 Let $\Gamma$ be the group of maps $\R^n\to\R^n$ of the form $x\mapsto Ax+b$, where $A\in\mathrm{SL}_n(\Z)$ and $b\in\R^n$.  A Delzant polytope $\Delta$ is  a convex polytope such that a neighbourhood of any vertex of $\Delta$ is equivalent via $\Gamma$ to a neighbourhood of the origin of the infinite polytope $(\R^+_0)^n$. Moreover, all vertices of the polytope are integers.
\end{dfn}
\begin{figure}
\begin{subfigure}{.5\textwidth}
  \centering
  \includegraphics[width=.8\linewidth]{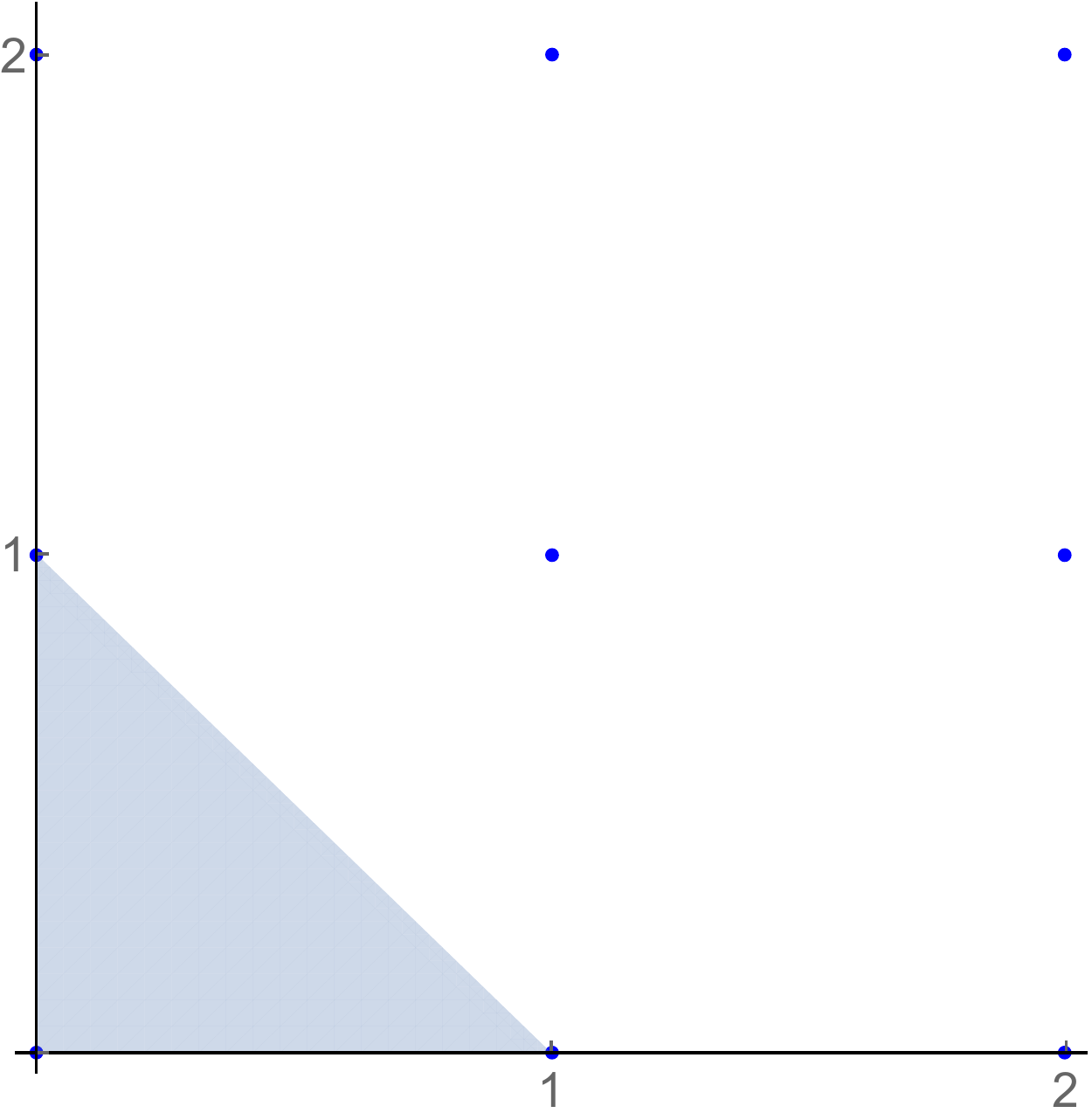}
\end{subfigure}%
\begin{subfigure}{.5\textwidth}
  \centering
  \includegraphics[width=.8\linewidth]{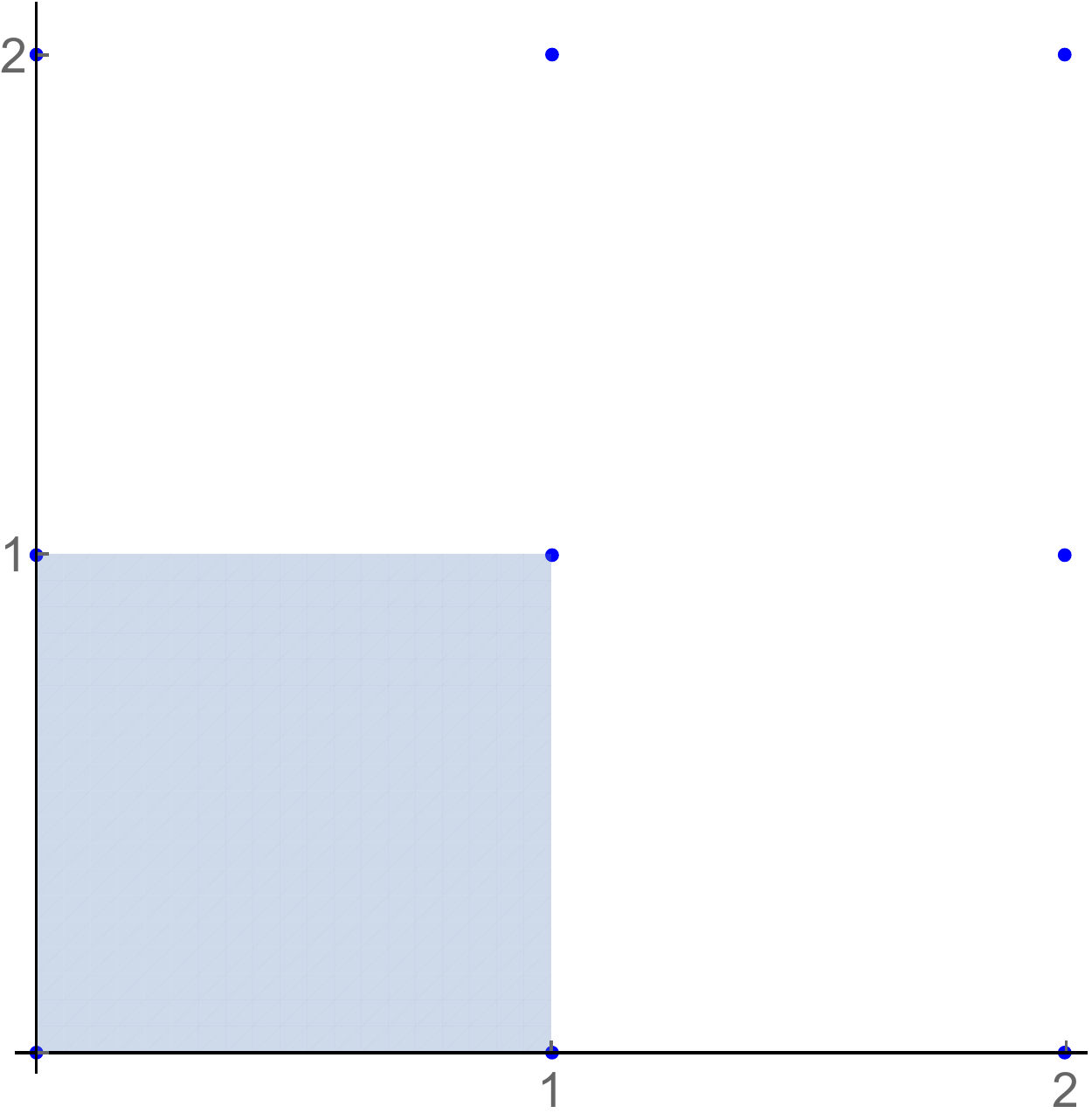}
\end{subfigure}
\begin{subfigure}{.5\textwidth}
  \centering
  \includegraphics[width=.8\linewidth]{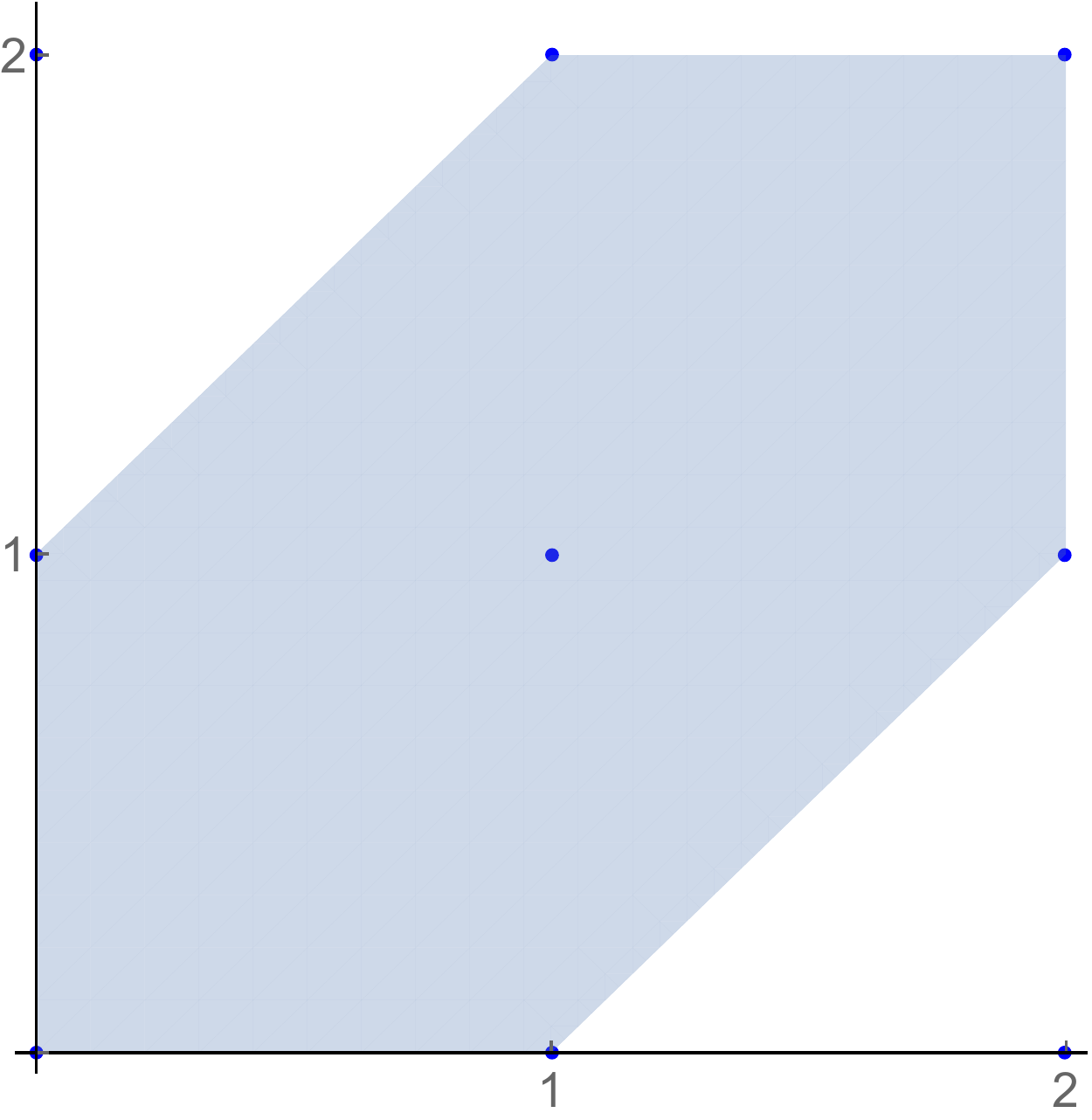}
\end{subfigure}
\caption{Examples of Delzant polytopes related to compact complex toric surfaces.}
\label{fig:fig}
\end{figure}

Vice versa, we have to recall that there exists a canonical way to associate  every  Delzant polytope with a compact smooth projective variety, containing a non-compact complex torus $(\C^*)^n$ as an open dense subset and whose natural action of $(\C^*)^n$ on itself can be holomorphically extended to the whole manifold (see e.g. \cite{don}). 
Moreover, Delzant showed in \cite{del} how such manifolds can be canonically endowed with a symplectic structure so that the trivial action of the real torus on that projective variety is Hamiltonian and moment imagine is the given Delzant polytope. Furthermore, he proved the existence of a bijective correspondence between compact symplectic toric manifolds (up to symplectomorphisms) and Delzant polytopes (up to transformations in $\Gamma$). 

Hence,  from now on, we can  consider without loss of generality \emph{only Delzant polytopes  having a vertex in the origin of $\R^n$ and its edges meeting at the origin agreeing  with the positive semi axes of $\R^n$} (see. e.g. Figure \ref{fig:fig}).\\

The following lemma has a key role in  Arezzo-Loi-Zuddas' proof.

\begin{lem}[C. Arezzo, A. Loi, F. Zuddas  \cite{ALZ}]\label{integer}
 Let $M$ be a compact  toric \KE manifold with Delzant polytope $\Delta$ and  endowed with a projectively  induced metric. Then, on the open dense subset of $M$ consisting of all maximal dimensional orbits of the toric action, is defined a local \K potential reading as \eqref{diastpol}, i.e.
 $$ \log\left(1+\sum_{j=1}^n|z_j|^2+\sum_{j=n+1}^N a_j|z^{m_j}|^2\right).$$
 A coefficient  $a_i$    is different from zero if and only if the corresponding multi-index $m_i$ belongs to $\Delta\cap\Z^n$.
\end{lem}

Compact toric Fano manifolds with dimension less than five have been completely classified (see e.g. \cite{batsel}). Up to homotheties,  Delzant polytopes of such manifolds admitting a \KE metric (different from  complex projective spaces with the Fubini-Study metric and that cannot be decomposed as a product of lower dimensional manifolds) are the following
\begin{itemize}
\item $\{(x,y)\in\R^2 |\ 0\leq x,y\leq 2,\ -1\leq x-y\leq1\} $;
\item $\{(x,y,z)\in\R^3|\ x,y\geq 0,\ 0\leq z\leq2,\ x+z\leq 3,\ y-z\leq 1 \} $;
\item $\{(x,y,z,w)\in\R^4|\ 0\leq x,y,z,w\leq 2,\ -1\leq x+y-z-w\leq 1 \} $;
\item $\{(x,y,z,w)\in\R^4|\ x,y,\geq 0,\ 0\leq z,w\leq2,\ -1\leq z-w\leq1,\ x-z\leq1,\ y+z\leq3 \} $;
\item $\{(x,y,z,w)\in\R^4|\ x,y,z,w\geq 0,\ x-z\leq 1,\ y-w\leq1,\ z+w\leq 3,\ x+y\leq1,\ z+w-x-y\leq 1 \} $.
\end{itemize}

On one hand, the Calabi's diastasis function of a projectively induced \KE metric on the previous toric manifold needs to assume the particular form stated by Lemma \ref{integer}. On the other, it should also be  a solution of the  Monge-Ampère equation \eqref{MA}. In \cite{ALZ},  the authors  showed via direct computations that in all the previous cases, these two fact lead to a contradiction.

Moreover, the authors proved that manifolds related to Delzant polytopes $k\Delta$, differing from the polytope $\Delta$ for a homothety of ratio $k$, can be endowed with a \KE projectively induced metric if and only if the  manifold related to Delzant polytope $\Delta$ can. Hence, compact toric Fano manifold with a \KE projectively induced metric cannot be different from a product of complex projective spaces with Fubini-Study metrics.

\subsection{Final remark}
The absence of a classification of compact Fano toric manifolds with complex dimension greater than four, precludes the possibility of extending the Arezzo-Loi-Zuddas' approach to higher dimensional cases.

In \cite{mannosalis}, G. Manno and F. Salis proposed a different approach, based on the classification of special polynomial solutions of the Monge-Ampère equation describing the Einstein condition. In this way, the authors  classified \inv projectively induces \KE manifolds with complex dimension two, independently from  Arezzo, Loi and Zuddas. It is an open problem whether this approach can be generalized to any dimension.

\small{}

\end{document}